\documentclass[reqno]{amsart}

\usepackage{amsmath}
\usepackage[all]{xy}
\usepackage{amssymb}
\usepackage{amscd}
\date{}
\begin{document}
\title{A Class of $J$-quasipolar Rings}
\author{M. B. Calci}
\address{Mete Burak Calci, Department of Mathematics, Ankara University,
Turkey} \email{mburakcalci@gmail.com}
\author{S. Halicioglu}
\address{Sait Halicioglu, Department of Mathematics, Ankara University,
Turkey} \email{halici@ankara.edu.tr}
\author{A. Harmanci}
\address{Abdullah Harmanci, Department of Mathematics, Hacettepe University,
Turkey} \email{harmanci@hacettepe.edu.tr}
\newtheorem{thm}{Theorem}[section]
\newtheorem{lem}[thm]{Lemma}
\newtheorem{prop}[thm]{Proposition}
\newtheorem{cor}[thm]{Corollary}
\newtheorem{df}[thm]{Definition}
\newtheorem{nota}{Notation}
\newtheorem{note}[thm]{Remark}
\newtheorem{ex}[thm]{Example}
\newtheorem{exs}[thm]{Examples}
\def\theequation{1. \arabic{equation}}
\begin{abstract}  In this paper, we introduce a class of $J$-quasipolar rings.
Let $R$ be a ring with identity. An element $a$ of a ring
$R$ is called {\it weakly $J$-quasipolar} if there exists $p^2 =
p\in comm^2(a)$ such that $a + p$ or $a-p$ are contained in $J(R)$
and the ring $R$ is called {\it weakly $J$-quasipolar} if every
element of $R$ is weakly $J$-quasipolar. We give many
characterizations and investigate general properties of weakly
$J$-quasipolar rings. If $R$ is a weakly $J$-quasipolar ring, then
we show that (1) $R/J(R)$ is weakly $J$-quasipolar, (2) $R/J(R)$
is commutative, (3) $R/J(R)$ is reduced. We use weakly
$J$-quasipolar rings to obtain more results for $J$-quasipolar
rings. We prove that the class of weakly $J$-quasipolar rings lies
between the class of $J$-quasipolar rings and the class of
quasipolar rings. Among others it is shown that a ring $R$ is
abelian weakly $J$-quasipolar if and only if $R$ is uniquely
clean.  \vspace{2mm}

\noindent {\bf2010 MSC:} 16S50, 16S70, 16U99

\noindent {\bf Key words:} Quasipolar ring, $J$-quasipolar ring,
weakly $J$-quasipolar ring, uniquely clean ring, feckly reduced
ring, directly finite ring
\end{abstract}
\maketitle
\maketitle
\section{Introduction}
Throughout this paper all rings are associative with identity
unless otherwise stated. Given a ring $R$, the symbol $U(R)$ and
$J(R)$ stand for the group of units and the Jacobson radical of
$R$, respectively.

 Let $R$ be a ring and
$a\in R$. We adopt the notations $comm(a) = \{b\in R\mid ab =
ba\}$ while the {\it second commutant} and $comm^2(a) = \{b\in
R\mid bc = cb $ for all $c\in comm(a)\}$ and $R^{qnil} = \{a\in
R\mid 1 + ax$ is invertible for each $x\in comm(a)\}$. An element
$a$ of a ring $R$ is called {\it quasipolar} (see \cite{KP}) if
there exists $p^2 = p\in R$ such that $p\in comm^2(a)$, $a + p\in
U(R)$ and $ap\in R^{qnil}$. Any idempotent $p$ satisfying the
above conditions is called a {\it spectral idempotent} of $a$, and
this term is borrowed from spectral theory in Banach algebra and
it is unique for $a$. Quasipolar rings have been studied by many
ring theorists (see \cite{CC},\cite{CC1}, \cite{KP} and
\cite{YC}). Recently, $J$-quasipolar rings are introduced in
\cite{CC3}. For an element $a$ of a ring $R$, if there exists
$p^2=p\in comm^2(a)$ such that $a+p\in J(R)$, then $a$ is called
\textit{$J$-quasipolar} and a ring $R$ is called $J$-quasipolar,
if every element of $R$ is $J$-quasipolar. It is proved that every
$J$-quasipolar ring is quasipolar.

Motivated by these classes of polarity versions of rings, we
introduce weakly $J$-quasipolar rings, generalizing $J$-quasipolar
rings.  Throughout this paper, some basic properties of weakly
$J$-quasipolar ring are studied, also examples and counter
examples are given. We show that the class of weakly
$J$-quasipolar rings lies properly between the class of
$J$-quasipolar rings and the class of quasipolar rings. It is
proved that $R$ is $J$-quasipolar if and only if $R$ is weakly
$J$-quasipolar and $2\in J(R)$. Then some of the main results of
$J$-quasipolar rings are special cases of our results for this
general setting.   Given a ring $R$, if $M_n(R)$ and $T_n(R)$
denote the ring of all $n\times n$ matrices and triangular
matrices over $R$, then we investigate necessary and sufficient
conditions as to weakly $J$-quasipolarity of $T_2(R)$ over a
commutative local ring $R$. Further, it is proven that $M_n(R)$ is
not weakly $J$-quasipolar for $n\geq 2$. Finally, we determine
under what conditions a $2\times 2$ matrix over a commutative
local ring is weakly $J$-quasipolar.

In what follows, $\Bbb{N}$ and $\Bbb{Z}$ denote the set of natural
numbers, the ring of integers and for a positive integer $n$,
$\Bbb{Z}_n$ is the ring of integers modulo $n$. The notations
$detA$ and $trA$ denote the determinant and the trace of a square
matrix $A$ over a commutative ring and $I_n$ denotes the $n\times
n$ identity matrix.

\section{Weakly $J$-Quasipolar Rings}

In this section, we introduce a class of quasipolar rings
 which  is a generalization of $J$-quasipolar rings.
 By using weakly
$J$-quasipolar rings, we obtain more results for $J$-quasipolar
rings. It is clear that every $J$-quasipolar ring is weakly
$J$-quasipolar and we supply an example to show that the converse
does not hold in general (see Example \ref{akornek}). Moreover, it
is shown that the class of weakly $J$-quasipolar rings  lies
strictly between the class of $J$-quasipolar rings and the class
of quasipolar rings (see Example \ref{akornek}, Corollary
\ref{cor.qua.1} and Example \ref{hoca}). We investigate general
properties of weakly $J$-quasipolar rings.

\begin{df}{\rm Let $R$ be a ring and $a\in R$. The element $a$ is called {\it
weakly $J$-quasipolar} if there exists $p^2 = p\in comm^2(a)$ such
that $a + p\in J(R)$ or $a-p\in J(R)$. The idempotent which
satisfies the above condition is called a {\it weakly J-spectral
idempotent} and $R$ is called {\it weakly $J$-quasipolar} if every
element of $R$ is weakly $J$-quasipolar.}
\end{df}

 Lemma \ref{mete} shows that weakly $J$-quasipolar elements and
 rings are abundant.

\begin{lem}\label{mete} Let $R$ be a ring. Then we have the followings.
\begin{enumerate}
\item Every idempotent  in $R$
is weakly $J$-quasipolar.
\item An element $a\in R$ is weakly $J$-quasipolar if and only if $-a\in R$ is weakly
$J$-quasipolar.
\item Every element in $J(R)$ is weakly $J$-quasipolar.
\item Boolean rings are weakly $J$-quasipolar.
\item $J$-quasipolar rings are weakly $J$-quasipolar.
\item Every uniquely clean ring is weakly $J$-quasipolar.
\end{enumerate}
\end{lem}

In the sequel,  we state elementary properties of weakly
$J$-quasipolar elements and weakly $J$-quasipolar rings.

\begin{lem}\label{uniqidempotent}
Let $R$ be a ring. If $u\in U(R)$ is weakly $J$-quasipolar, then
$1$ is the weakly $J$-spectral idempotent of u.
\end{lem}

\begin{proof}
Let $u\in U(R)$ be weakly $J$-quasipolar, so $u+p\in J(R)$ or
$u-p\in J(R)$ such that $p^2=p\in comm^2(u)$. If $u-p\in J(R)$,
then $u^{-1}u-u^{-1}p=1-u^{-1}p\in J(R)$. Hence, $u^{-1}p\in U(R)$
and so $p\in U(R)$. Thus, we have $p=1$. In case $u+p \in J(R)$,
the proof is similar.
\end{proof}

 By using the concept of $J$-quasipolarity, we obtain a characterization for local rings.

\begin{prop}
Let $R$ be a weakly $J$-quasipolar ring. Then $R$ is a local ring
if and only if $R$ has only trivial idempotents.
\end{prop}

\begin{proof}
Assume that $R$ is a weakly $J$-quasipolar ring and has only
trivial idempotents. Let $a\in R$, so $a+1\in J(R)$ or $a-1\in
J(R)$ or $a\in J(R)$. If $a+1\in J(R)$ or $a-1\in J(R)$, then
$a\in U(R)$. In the last condition, $a\in J(R)$. Consequently, $R$
is a local ring. The converse statement is clear.
\end{proof}

\begin{lem}\label{benzer.eleman}
Let $R$ be a ring. If $a\in R$ and $u\in U(R)$,  then $a$ is
weakly $J$-quasipolar if and only if $u^{-1}au$ is weakly
$J$-quasipolar.
\end{lem}

\begin{proof}
Assume that $a$ is weakly $J$-quasipolar. Then there exists
$p^2=p\in comm^2(a)$ such that $a-p\in J(R)$. If $q$ is taken as
$q=u^{-1}pu$, then $q^2=q\in R$ and
$u^{-1}au-u^{-1}pu=u^{-1}(a-p)u\in J(R)$. Let $b\in
comm(u^{-1}au)$, then $(u^{-1}au)b=b(u^{-1}au)$ and so
$a(ubu^{-1})=(ubu^{-1})a$. Thus $ubu^{-1}\in comm(a)$. Since $p\in
comm^2(a)$, $(ubu^{-1})p=p(ubu^{-1})$. Hence
$b(u^{-1}pu)=(u^{-1}pu)b$ and $pb=bp$. Consequently, $p\in
comm^2(u^{-1}au)$ and so $u^{-1}au$ is weakly $J$-quasipolar.
Conversely, assume that $u^{-1}au-p\in J(R)$, so $a-uqu^{-1}\in
J(R)$. Also $(uqu^{-1})^2=uqu^{-1}\in comm^2(a)$. If $a+p\in
J(R)$, then proof is similar.
\end{proof}

The proof of  Lemma \ref{benzer.eleman} reveals that $p$ is weakly
$J$-spectral idempotent of $a$ if and only if $u^{-1}pu$ is the
weakly $J$-spectral idempotent of $u^{-1}au$. We need the
following lemma in order to prove Theorem \ref{eRe}.

\begin{lem}\label{burcu}
Let $R$ be a ring. If $a=j_1-p\in J(R)$ or $a=j_2+p\in J(R)$ is
weakly $J$-quasipolar decomposition of $a$ in $R$, then
$ann_l(a)\subseteq ann_l(p)$ and $ann_r(a)\subseteq ann_r(p)$.
\end{lem}

\begin{proof}
If $r\in ann_l(a)$, then $ra=0$. Assume that $a+p=j_1\in J(R)$
such that $p^2=p\in comm^2(a)$ and also $ap=pa$. Then
$rp=r(j_1-a)=rj_1$ and so $rp=rj_1p=rpj_1$. Hence
$rp(1-j_1)=rp-rpj_1=0$. Since $1-j_1\in U(R)$, $r\in ann_l(p)$. If
$r\in ann_r(a)$, then $ar=0$. Thus $pr=(j_1-a)r=j_1r$ and so
$pr=pj_1r=j_1pr$. Also $(1-j_1)pr=pr-j_1pr=0$. Because of
$1-j_1\in U(R)$, $r\in ann_r(p)$. If $a-p=j_2\in J(R)$ such that
$p^2=p\in comm^2(a)$, then the proof is similar to above.
\end{proof}

\begin{thm}\label{eRe}
If $R$ is weakly $J$-quasipolar, then so is $fRf$ for all
$f^2=f\in R$.
\end{thm}

\begin{proof}
For every $a\in fRf$ there exists  $p\in comm^2(a)$ such that
$a-p\in J(R)$ or $a+p\in J(R)$. Let $a+p=j_1\in J(R)$ or
$a-p=j_2\in J(R)$. By Lemma \ref{burcu}, we have $1-f\in
ann_l(a)\cap ann_r(a)\subseteq ann_l(p)\cap ann_r(p)=R(1-p)\cap
(1-p)R=(1-p)R(1-p)$. Then $pf=p=fp$ and so $a=fjf-fpf$,
$(fpf)^2=fpf$ and $fjf\in fJ(R)f=J(fRf)$. Lastly, let $xa=ax$ and
$x\in fRf$, so $x (fpf)=(fpf)x$. If $a-p=j_2\in J(R)$, then proof
is similar. Consequently, $a$ is weakly $J$-quasipolar in $fRf$.
\end{proof}

By the definition of weakly $J$-quasipolar rings, it is clear that
every $J$-quasipolar ring is weakly $J$-quasipolar. We now
investigate under what condition a weakly $J$-quasipolar ring is
$J$-quasipolar.

\begin{prop}\label{2injr}
A ring $R$ is $J$-quasipolar if and only if $R$ is weakly
$J$-quasipolar and $2\in J(R)$.
\end{prop}

\begin{proof}
Let $R$ be a weakly $J$-quasipolar ring and $2\in J(R)$. If
$a+p\in J(R)$ such that $p^2=p\in comm^2(a)$, then it is clear.
Let $a-p\in J(R)$ and $p^2=p\in comm^2(a)$. Since $2\in J(R)$,
$a-p+2p\in J(R)$ and so $a$ is $J$-quasipolar. The converse is
clear.
\end{proof}

The next example illustrates that there are weakly $J$-quasipolar
rings but not $J$-quasipolar.

\begin{ex}\label{akornek} The ring $\Bbb Z_6$ is weakly $J$-quasipolar but not
$J$-quasipolar.
\end{ex}

\begin{proof}  It is obvious that $\Bbb Z_6$ is weakly
$J$-quasipolar. Since $1+1\notin J(\Bbb Z_6)=0$,  by Proposition
\ref{2injr}, $\Bbb Z_6$ is  not $J$-quasipolar.
\end{proof}

In \cite{CC3}, it is shown that every $J$-quasipolar element is
quasipolar. We obtain the following result for this general
setting.

\begin{prop}\label{quasipolar1}
Every weakly $J$-quasipolar element in a ring $R$ is quasipolar.
\end{prop}

\begin{proof}
Let $a\in R$ be weakly $J$-quasipolar. Then there exists $p^2=p\in
comm^2(a)$ such that $a+p\in J(R)$ or $a-p\in J(R)$. If $a+p\in
J(R)$, then $a$ is quasipolar from \cite[Proposition 2.4]{CC3}. If
$a-p\in J(R)$ such that $p^2=p\in comm^2(a)$, then $a+(1-p)\in
U(R)$ and also $(a-p)(1-p)=a(1-p)\in J(R)\subseteq R^{qnil}$.
Therefore $a$ is a quasipolar element.
\end{proof}

\begin{cor}\label{cor.qua.1}
If $R$ is weakly $J$-quasipolar, then it is quasipolar.
\end{cor}

The converse statement of Corollary \ref{cor.qua.1} is not true in
general, i.e., there are quasipolar rings but not weakly
$J$-quasipolar.

\begin{ex}\label{hoca}
Let $R=\Bbb Z_{(5)}$ be the localization ring of $\Bbb Z$ at the
prime $5$. Then $R$ is a local ring and thus quasipolar by \cite
[Corollary 3.3]{YC}. Since $\frac{1}{3}\in \Bbb Z_{(5)}$ is not
weakly $J$-quasipolar,  $\Bbb Z_{(5)}$ is not weakly
$J$-quasipolar. \end{ex}

By Example \ref{akornek}, Corollary \ref{cor.qua.1} and Example
\ref{hoca}, it is clear that the class of weakly $J$-quasipolar
rings  lies strictly between the class of $J$-quasipolar rings and
the class of quasipolar rings.

\begin{prop}\label{idem.teklik}
Any weakly $J$-quasipolar element $a\in R$ has a unique weakly
$J$-spectral idempotent.
\end{prop}

\begin{proof}
Assume that $p, q$ are weakly $J$-spectral idempotents of $a \in
R$.\\
\textbf{Case 1:} If $a+p \in J(R)$ and $a+q \in J(R)$, then $1-p$
and $1-q$ are spectral idempotents of $-a$ by the proof of
Proposition \ref{quasipolar1}. By \cite{CC3}, the spectral
idempotent of $a$ and $-a$ is equal. Also by \cite[Proposition 2.3
]{KP}, the spectral idempotent of $a$ is unique, so we obtain that
$1-p =
1-q$. Then p = q.\\
\textbf{Case 2:} Assume that $a+p \in J(R)$ and $a-q \in J(R)$.
Then $1-p$ is spectral idempotent of $-a$ and $1-q$ is spectral
idempotent of
$a$. The remaining proof is similar to Case 1.\\
\textbf{Case 3:} Assume that $a-p \in J(R)$ and $a+q \in J(R)$,
then
similarly $p=q$.\\
\textbf{Case 4:} Assume that $a-p \in J(R)$ and $a-q \in J(R)$,
then similarly $p=q$.
\end{proof}

In \cite{Ch}, an element of a ring is called {\it strongly
$J$-clean } provided that it can be written as the sum of an
idempotent and an element in its Jacobson radical that commute. A
ring is {\it strongly $J$-clean} in case each of its elements is
strongly $J$-clean. From the definition of a strongly $J$-clean
ring, one may suspects that every weakly $J$-quasipolar ring is
strongly $J$-clean. But the following example erases possibility.

\begin{ex} It is clear that the ring $\Bbb Z_3$ is weakly
$J$-quasipolar. Since $2-1\notin J(\Bbb Z_3)$, it is not strongly
$J$-clean.
\end{ex}

Recall that, a ring R is called {\it periodic} if for each $x \in
R$, there exists distinct positive integers $m,n$ depending on x,
for which $x^n = x^m$.  For an easy reference, we mention Lemma
\ref{jacobsonthm} which is one of Jacobson's theorem given in
\cite{Jac} relating to periodicity and commutativity of the rings.

\begin{lem}\label{jacobsonthm}
Let $R$ be a ring in which for every $a \in R$ there exists an
integer $n(a) > 1$, depending on a such that $a^{n(a)} = a$,  then
$R$ is commutative.
\end{lem}

We now give a useful result to determine whether $R$ is weakly
$J$-quasipolar.

\begin{thm}\label{thmcomm}
If a ring $R$ is weakly $J$-quasipolar, then $R/J(R)$ is a
periodic ring which has three period and $R/J(R)$ is commutative.
\end{thm}

\begin{proof}
Let $R$ be weakly $J$-quasipolar and $r\in R$. So $r+p\in J(R)$ or
$r-p\in J(R)$ such that $p^2=p\in comm^2(a)$. Clearly, $\overline
r=\overline p$ or $\overline r=-\overline p$ and $\overline
p^2=\overline p$. If $\overline r=\overline p$, then $\overline
r^2=\overline r$ and so $\overline r^3=\overline r$. If $\overline
r=-\overline p$, then it is clear that $\overline r^3=\overline
r$. Hence $R/J(R)$ is a periodic ring which has period three. By
Lemma \ref{jacobsonthm},  $R/J(R)$ is commutative.
\end{proof}

The following example shows that the converse statement of Theorem
\ref{thmcomm} is not true in general.

\begin{ex} It is clear that the ring $\Bbb Z$ is commutative, $J(\Bbb Z)=0$ and $\Bbb Z/J(\Bbb Z)\cong
\Bbb Z$. But $\Bbb Z$ is not weakly $J$-quasipolar.
\end{ex}

By Theorem \ref{thmcomm}, we obtain the following important result
for weakly $J$-quasipolar rings.

\begin{cor}
If $R$ is weakly $J$-quasipolar, then $R/J(R)$ is weakly
$J$-quasipolar.
\end{cor}

\begin{proof}
Let $\overline x=r+J(R)$ and $R$ is weakly $J$-quasipolar. Hence
$r+p=j_1\in J(R)$ or $r-p=j_2\in J(R)$ such that $p^2=p\in
comm^2(r)$. Thus $r=j_1-p$ or $r=j_2+p$ and so $\overline
x=-p+J(R)$ or $\overline x=p+J(R)$. $R/J(R)$ is commutative by
Theorem \ref{thmcomm}, so $R/J(R)$ is weakly $J$-quasipolar from
Lemma \ref{mete} (1).
\end{proof}

Recall that a ring $R$ is said to be {\it clean} if for each $a\in
R$ there exists $e^2 = e\in R$ such that $a-e\in U(R)$. According
to Nicholson and Zhou \cite{NZ}, a ring $R$ is said to be {\it
uniquely clean} if for each $a\in R$ there exists unique
idempotent $e\in R$ such that $a -e\in U(R)$. In \cite{CC3}, it is
proved that a ring $R$ is uniquely clean if and only if $R$ is
abelian (i.e., each idempotent of $R$ is central) $J$-quasipolar.
In this direction we generalize this result for weakly
$J$-quasipolar rings.

\begin{thm}\label{uniquely.clean}
A ring $R$ is abelian weakly $J$-quasipolar if and only if $R$ is
uniquely clean.
\end{thm}

\begin{proof}
Assume that $R$ is a uniquely clean ring. By \cite[Lemma 4]{NZ},
$R$ is abelian. Since \cite[Teorem 20]{NZ}, for every $a\in R$
there exists unique $p^2=p\in R$ such that $e-a\in J(R)$. So
$a-e\in J(R)$ and $R$ is abelian weakly $J$-quasipolar.
Conversely, given $a\in R$, then $-a\in R$. Hence $-a+p\in J(R)$
or $-a-p\in J(R)$ such that $p^2=p\in R$. If $-a+p\in J(R)$, so
$a$ is uniquely clean. If $-a-p\in J(R)$, then $a-(1-p)\in
U(R)$.Uniqueness of the idempotent $p$ follows from
\ref{idem.teklik}. Therefore $R$ is a uniquely clean ring.
\end{proof}

The next example illustrates that ``abelian" condition is not
superfluous in Teorem \ref{uniquely.clean}.

\begin{ex} The matrix ring $T_2(\Bbb Z_2)$ is weakly $J$-quasipolar, but not abelian. By
\cite[Lemma 4]{NZ},  $T_2(\Bbb Z_2)$ is not a uniquely clean ring.
\end{ex}

In \cite{Uohh}, Ungor et al. introduced and studied a new class of
reduced rings (i.e., it has no nonzero nilpotent elements). A ring
$R$ is called {\it feckly reduced} if $R/J(R)$ is a reduced ring.
In this direction we show that every weakly $J$-quasipolar ring is
feckly reduced.

\begin{thm}\label{sait}
If $R$ is a weakly $J$-quasipolar ring, then it is feckly reduced.
\end{thm}

\begin{proof}
Let $R$ be weakly $J$-quasipolar and $r^2=0$. Therefore there
exists $p^2=p\in comm^2(r)$ such that $r+p\in J(R)$ or $r-p\in
J(R)$. If $r-p\in J(R)$, then $r(r-p)=r^2-rp\in J(R)$. Since
$r^2=0\in J(R)$, $rp\in J(R)$. Also $(r-p)p=rp-p\in J(R)$. Hence
$p\in J(R)$ and so $p=0$. Thus $r\in J(R)$ and $R/J(R)$ is
reduced. If $r+p\in J(R)$, then similarly $r\in J(R)$ and $R/J(R)$
is reduced. Consequently, $R$ is a feckly reduced ring.
\end{proof}

Let $J^\sharp(R)$ denote the subset $\{x \in R~|~ \exists n\in
\Bbb N ~such ~that~ x^n \in J(R)\}$ of $R$. It is obvious that
$J(R) \subseteq J^\sharp(R)$. Weakly $J$-quasipolar rings play an
important role for the reverse inclusion.

\begin{cor}\label{diyez}
If $R$ is a weakly $J$-quasipolar ring, then $J(R)=J^{\sharp} (R)$
\end{cor}

\begin{proof} Let $R$ be a weakly $J$-quasipolar ring. By Theorem
\ref{sait},  $R$ is feckly reduced and so $J(R)=J^{\sharp} (R)$
from \cite[Proposition 2.6]{Uohh}.
\end{proof}

The following result follows from Corollary  \ref{diyez}.

\begin{cor}
If $R$ is a  $J$-quasipolar ring, then  $J(R)=J^{\sharp} (R)$.
\end{cor}

Corollary \ref{diyez} is helpful to show that a ring is not weakly
$J$-quasipolar.

\begin{ex}
Let $R$ denote the ring $M_2(\Bbb Z_2)$. Then
$$J^{\sharp}(R)=\left\{
\begin{bmatrix}0&0\\0&0\end{bmatrix},
\begin{bmatrix}0&1\\0&0\end{bmatrix},\begin{bmatrix}0&0\\1&0\end{bmatrix},
\begin{bmatrix}1&1\\1&1\end{bmatrix}\right\}$$
and $J(R)=\left\{\begin{bmatrix}0&0\\0&0\end{bmatrix}\right\}$. By
Corollary \ref{diyez}, $R$ is not weakly $J$-quasipolar.
\end{ex}

Let $R$ be a ring and $a,b\in R$. Then $R$ is called
\textit{directly finite}, if $ab=1$ then $ba=1$. It is well known
that  $R$ is directly finite  if and only if $R/J(R)$ is directly
finite.

\begin{prop}\label{dirfin} If a ring $R$ is weakly $J$-quasipolar, then $R$ is directly finite.
\end{prop}

\begin{proof}
The proof is clear from \cite[Proposition 4.8]{Uohh}.
\end{proof}

Since every $J$-quasipolar ring is weakly $J$-quasipolar, the
following result follows from Proposition \ref{dirfin}.

\begin{cor}
If $R$ is a  $J$-quasipolar ring, then $R$ is directly finite.
\end{cor}

In \cite{Ni}, strongly clean rings are introduced and studied. A
ring $R$ is \textit{strongly clean}, if for every $a\in R$ there
exists $e^2=e\in R$ such that $a-e\in U(R)$ and $ae=ea$. At the
end of that paper, the authors ask some open questions. One of
them is `` Is every strongly clean ring directly finite?". By
Proposition \ref{dirfin}, weakly $J$-quasipolar rings are both
strongly clean and directly finite.

We also note that  it is easy to show that for $n\geq 2$,  the
matrix ring $M_n(R)$ is not weakly $J$-quasipolar from Proposition
\ref{dirfin}.

\section {Weakly $J$-quasipolarity of Matrix rings}

At the end of the second section of this paper, we show that any
matrix ring $M_n(R)$ with $n\geq 2$ is not weakly $J$-quasipolar.
It is still important to determine whether an individual matrix is
weakly $J$-quasipolar. In this section we study weakly
$J$-quasipolarity of some matrix rings. In particular, we
investigate necessary and sufficient conditions weakly
$J$-quasipolarity of the matrix ring $T_2(R)$ over a commutative
local ring $R$. We determine under what conditions a single
$2\times 2$ matrix over a commutative local ring is weakly
$J$-quasipolar.

We start with the obvious proposition.

\begin{prop}\label{ýdempotent} {\rm (1)} Let $R$ be a commutative local ring. Then $A\in M_2(R)$ is an
idempotent if and only if either $A = 0$, or $A = I_2$, or $A =
\begin{bmatrix}a&b\\
c&1-a\end{bmatrix}$ where $bc = a-a^2$.\\
{\rm (2)}  Let $R$ be a commutative local ring and $P\in T_2(R)$.
Then $P$ is an idempotent if and only if $P$ has a form
 $\begin{bmatrix}$
$1     &0      \\$ $ 0      &1$ $\end{bmatrix}$,
 $\begin{bmatrix}$
$0     &0      \\$ $ 0      &0$ $\end{bmatrix}$,
 $\begin{bmatrix}$
$1     &x      \\$ $ 0      &0$ $\end{bmatrix}$,
 $\begin{bmatrix}$
$0     &x      \\$ $ 0      &1$ $\end{bmatrix}$ for some $x \in
R$.
\end{prop}

\begin{proof} (1) is clear from \cite[Lemma 16.4.10]{Ch2} and (2)
is straightforward.
\end{proof}

\begin{prop}\label{ustucgen}
Let $R$ be a commutative local ring. $A$=
 $\begin{bmatrix}$
  $a_1     &  a_2      \\$

 $ 0      &  a_3$
$\end{bmatrix}$ is weakly $J$-quasipolar in $T_2(R)$ if and only
if one of the following holds:\begin{enumerate}
\item $A\in J(T_2(R))$,
\item $A\in U(T_2(R))$,
\item $A+P$ or $A-P\in J(T_2(R))$ where $P=$ $\begin{bmatrix}$
  $1     &  x      \\$
$ 0      &  0 $ $\end{bmatrix}$ such that $x=(a_1-a_3)^{-1}a_2$,
\item $A-P$ or
$A+P\in J(T_2(R))$ where $P=$
 $\begin{bmatrix}$
  $0     &  x      \\$

 $ 0      &  1 $
$\end{bmatrix}$ such that $x=(a_3-a_1)^{-1}a_2$.
\end{enumerate}

\end{prop}

\begin{proof}
Assume that $A$ is weakly $J$-quasipolar.\\
\textbf{Case 1:} Let $A+P\in J(T_2(R))$ such that $P^2=P\in
comm^2(A)$. Since $A+P=$
 $\begin{bmatrix}$
  $a_1+p_1     &  a_2+p_2      \\$

 $ 0      &  a_3+p_3$
$\end{bmatrix}$ $\in J(T_2(R))$, $a_1+p_1\in J(R)$ and $a_3+p_3\in
J(R)$. Besides assume that $B\in comm(A)$ and take $B=$
 $\begin{bmatrix}$
  $b_1     &  b_2      \\$

 $ 0      &  b_3$
$\end{bmatrix}$, so \\ $\begin{bmatrix}$
  $b_1a_1     &  b_1a_2+b_2a_3      \\$

 $ 0      &  b_3a_3$ $\end{bmatrix}$=$\begin{bmatrix}$
  $a_1b_1     &  a_1b_2+a_2b_3      \\$

 $ 0      &  a_3b_3$
 $\end{bmatrix}$. Therefore $a_2(b_1-b_3)=b_2(a_1-a_3)$. \\
\textbf{(i)} If $a_1,a_3\in J(R)$, then $p_1=p_3=0$. Hence
$p_2=0$.\\
\textbf{(ii)} If $a_1,a_3\in U(R)$, then $p_1=p_3=1$. Hence
$p_2=0$.\\
\textbf{(iii)} If $a_1\in J(R)$, $a_3\in U(R)$, then $p_1=0,p_3=1$
and $p_2=x\in R$. Since $a_1-a_3\in U(R)$,
$b_2=(a_1-a_3)^{-1}a_2(b_1-b_3)$. Providing $x=(a_3-a_1)^{-1}a_2$, then $P\in comm(B)$. Hence $P\in comm^2(A)$.\\
\textbf{(iv) }If $a_1\in U(R)$, $a_3\in J(R)$, then $p_1=1,p_3=0$
and $p_2=x\in R$. Because of $a_1-a_3\in U(R)$,
$b_2=(a_1-a_3)^{-1}a_2(b_1-b_3)$. Providing
$x=(a_1-a_3)^{-1}a_2$, then $P\in comm(B)$. Therefore $P\in comm^2(A)$. \\
\textbf{Case 2:} Let $A-P\in J(T_2(R))$ such that $P^2=P\in
comm^2(A)$. Proof is similar to proof of Case 1.

The converse statement is clear.
\end{proof}

The following result is a direct consequence of Proposition
\ref{ustucgen} for $J$-quasipolar rings.

\begin{cor}
Let $R$ be a commutative local ring. $A$=
 $\begin{bmatrix}$
  $a_1     &  a_2      \\$

 $ 0      &  a_3$
$\end{bmatrix}$ is $J$-quasipolar in $T_2(R)$ if and only if one
of the following holds:\begin{enumerate}
\item $A\in J(T_2(R))$.
\item $A\in U(T_2(R))$.
\item $A+P\in J(T_2(R))$ where $P=$ $\begin{bmatrix}$
  $1     &  x      \\$

 $ 0      &  0 $
$\end{bmatrix}$ such that $x=(a_1-a_3)^{-1}a_2$ or
$x=(a_3-a_1)^{-1}a_2$.
\end{enumerate}
\end{cor}

\begin{cor}\label{ust3}
Let $R$ be a ring. If  $T_n(R)$ with $n\geq 2$ is weakly
$J$-quasipolar, then $R$ is weakly $J$-quasipolar.
\end{cor}

\begin{proof}
Assume that $T_n(R)$ is weakly $J$-quasipolar. Let $f$ be the unit
matrix with $(1,1)$ entry is $1$ and the other entries are $0$,
then $fT_n(R)f\cong R$.  By Theorem \ref{eRe},  $R$ is weakly
$J$-quasipolar. \end{proof}

The following example illustrates that the converse statement of
Corollary \ref{ust3} is not true in general.

\begin{ex}\label{yosum} If $R=\Bbb Z_3$, then $R$ is weakly $J$-quasipolar.   For $A$=$\begin{bmatrix}$
  $1     &  0      \\$

 $ 0      &  2$
$\end{bmatrix}$ $\in U(T_2(R))$, $A+I_2 \notin J(T_2(R))$ and
 $A-I_2  \notin J(T_2(R))$.
Therefore $T_2(R)$ is not weakly $J$-quasipolar.
\end{ex}

Our next endeavor is to find conditions under which an individual
matrix in $M_2(R)$ is weakly $J$-quasipolar.

\begin{lem}
Let $R$ be a ring. Then $A\in U(M_2(R))$ and $A$ is weakly
$J$-quasipolar if and only if $A-I_2\in J(M_2(R))$ or $A+I_2\in
J(M_2(R))$.
\end{lem}

\begin{proof}
Let $A$ be weakly $J$-quasipolar. Since $A\in U(M_2(R))$, weakly
$J$-spectral idempotent of $A$ is $I_2$. \\\textbf{Case 1:} If
$A+I_2\in
J(R)$, then it is clear from proof of \cite[Lemma 4.3]{CC3}.\\
\textbf{Case 2:} If $A-I_2\in J(R)$, then $A\in
1+J(M_2(R))\subseteq U(M_2(R))$.\\ The converse is clear.
\end{proof}

The following lemma is important to study especially in a matrix
ring.

\begin{lem}\label{altýinjr}
If $R$ is a weakly $J$-quasipolar ring, then $6\in J(R)$.
\end{lem}

\begin{proof}
Let $R$ be a weakly $J$-quasipolar ring, then there exists
$p^2=p\in comm^2(2)$ such that $2-p\in J(R)$ or $2+p\in J(R)$.
Assume that $2-p=j\in J(R)$, therefore $2-j=p$ and $(2-j)^2=2-j$.
Thus $2=j(3-j)\in J(R)$. As a consequence $6\in J(R)$. If
$2+p=j_1\in J(R)$, then $(j_1-2)^2=(j_1-2)$. So $6=j_1(5-j_1)\in
J(R)$.
\end{proof}

The following result is helpful to show a ring is not weakly
$J$-quasipolar.

\begin{ex}  Since $6\notin J(\Bbb Z_{15})=0$, by Lemma \ref{altýinjr},
$\Bbb Z_{15}$ is  not weakly $J$-quasipolar.
\end{ex}

The converse statement of Lemma \ref{altýinjr} is not true in
general, i.e., for a ring $R$, if  $6\in J(R)$, then $R$ need not
be weakly $J$-quasipolar.

\begin{ex} It is obvious that $6\in J(T_2(\Bbb Z_3))$.  By Example \ref{yosum}, the ring $T_2(\Bbb Z_3)$ is
not weakly $J$-quasipolar.
\end{ex}

Proposition \ref{2injr} shows that in case of $2\in J(R)$, weakly
$J$-quasipolar rings and $J$-quasipolar rings are the same. The
following example indicates that it does not hold in case of $6\in
J(R)$.

\begin{ex} The ring $\Bbb Z_9$ is weakly $J$-quasipolar and $6\in J(\Bbb
Z_9)$. Since there is not a $J$-spectral idempotent for $4$ such
that $4+p\in J(\Bbb Z_9)$, it is not $J$-quasipolar.
\end{ex}

\begin{lem}\label{aartýeksibes}
Let $R$ be a ring with $6\in J(R)$. If $a\in R$ is weakly
$J$-quasipolar, then $a+5$ or $a-5$ is weakly $J$-quasipolar.
\end{lem}

\begin{proof}
Let $a \in R$ be weakly $J$-quasipolar. Thus $a+p\in J(R)$ or
$a-p\in J(R)$ such that $p^2=p\in comm^2(a)$. Assume that $a+p\in
J(R)$ and $p^2=p\in comm^2(a)$. Since $6\in J(R)$,
$a-6+p=(a-5)-(1-p)\in J(R)$. So $a-5$ is weakly $J$-quasipolar. If
$a-p\in J(R)$ such that $p^2=p\in comm^2(a)$,
$a+6-p=(a+5)+(1-p)\in J(R)$.
\end{proof}

\begin{prop}
Let $R$ be a commutative ring with $6\in J(R)$ and $A\in M_2(R)$
such that $A\notin J(M_2(R))$. If both $detA$ and $trA$ in $J(R)$,
then $A$ is not weakly $J$-quasipolar.
\end{prop}

\begin{proof} If $A$ is weakly $J$-quasipolar, then $A-5$ or $A+5$ weakly
$J$-quasipolar by Lemma \ref{aartýeksibes}. Note that
$det(A-5)=detA-5(trA+5)\in U(R)$. Hence weakly $J$-spectral
idempotent of $A-5$ is $I_2$ by Lemma \ref{uniqidempotent}. So
$A-5-I_2\in J(M_2(R))$ or $A-5+I_2\in J(M_2(R))$. If $A-5-I_2\in
J(M_2(R))$, then $A$ is weakly $J$-quasipolar, which contradicts
the assumption $A\notin J(M_2(R))$. In other condition, let
$A-5+I_2\in J(M_2(R))$ and so $A-4\in J(M_2(R))$. Therefore
$a_{11}-4,a_{22}-4\in J(R)$, $a_{11}+a_{22}-8=trA-8\in J(R)$.
Since $trA\in J(R)$, so $8\in J(R)$ and $8-6=2\in J(R)$. Thus
$A-4+4\in J(M_2(R))$ is a contradiction. As a consequence $A$ is
not weakly $J$-quasipolar. Also in case of $A+5\in J(M_2(R))$,
proof is similar. Finally $A$ is not weakly $J$-quasipolar.
\end{proof}

\begin{lem}
Let $R$ be a commutative local ring. Then $A= \left[%
\begin{array}{cc}
  j & 0 \\
  0 & u \\
\end{array}%
\right]$ is weakly $J$-quasipolar in $M_2(R)$ if and only if one
of the following holds.
\begin{enumerate}
\item $A \in J(M_2(R))$.
\item $A+I_2 \in J(M_2(R))$.
\item $A-I_2 \in J(M_2(R))$.
\item $u \in -1+J(R)$ and $j\in J(R)$.
\item $u\in J(R)$ and $j\in -1+J(R)$.
\item $u\in J(R)$ and $j\in 1+J(R)$.
\item $u \in 1+J(R)$ and $j\in J(R)$.
\end{enumerate}
\end{lem}

\begin{proof}
 Let $A$ be weakly $J$-quasipolar. Then, there exists $P^2=P \in
comm^2(A)$ such that $A+P \in J(M_2(R))$ or $A-P \in J(M_2(R))$.
If $A+P \in J(M_2(R))$, then $(1), (2), (4),(5)$ hold by
\cite[Lemma 4.7]{CC3}. Assume that $A-P \in J(M_2(R))$. If $P=0$
or $P=I_2$ it is clear. Let $P\neq 0$ and $P \neq I_2$. By
Proposition \ref{ýdempotent}, $P =
\begin{bmatrix}a&b\\c&1-a\end{bmatrix}$ where $bc = a-a^2$. Since
$F=\begin{bmatrix}1&0\\0&0\end{bmatrix} \in comm(A)$ and $P \in
comm^2(A)$, $FP=PF$. Then, $b=c=0$. Thus, $P =
\begin{bmatrix}1&0\\0&0\end{bmatrix}$ or $P =
\begin{bmatrix}0&0\\0&1\end{bmatrix}$. Since $A-P \in J(M_2(R))$,
$u\in J(R)$ and $j\in 1+J(R)$ or $u \in 1+J(R)$ and $j\in J(R)$.\\
Conversely, if $A \in J(M_2(R))$ or $A+I_2 \in J(M_2(R))$ or
$A-I_2 \in J(M_2(R))$, then $A$ is weakly $J$-quasipolar. If $u
\in -1+J(R)$ and $j\in J(R)$ or $u\in J(R)$ and $j\in -1+J(R)$,
then it follows from \cite[Lemma 4.7]{CC3}. Suppose that $u\in
J(R)$ and $j\in 1+J(R)$. Let $P =
\begin{bmatrix}1&0\\0&0\end{bmatrix}$. Then $P^2=P$ and $A-P \in
J(M_2(R))$. To show that $P^2=P \in comm^2(A)$, let $B=
\begin{bmatrix}x&y\\z&t\end{bmatrix} \in comm(A)$. Hence $y=z=0$
and so $PB=BP$. Thus $A$ is weakly $J$-quasipolar. If $u\in J(R)$
and $j\in 1+J(R)$, similarly $A$ is weakly $J$-quasipolar.
\end{proof}

\begin{prop}
Let $R$ be a commutative local ring with $6\in J(R)$ and let $A\in
M_2(R)$ such that $A\notin J(M_2(R))$ and $detA\in J(R)$. Then $A$
is weakly $J$-quasipolar if and only if $x^2-(trA)x+detA=0$ has a
root in $J(R)$ and a root in $\mp 1+J(R)$.
\end{prop}

\begin{proof}
The proof is similar to proof of \cite[Proposition 4.8]{CC3}.
\end{proof}

\begin{thm}
Let $R$ be a commutative local ring with $6\in J(R)$. The matrix
$A\in M_2(R)$ is weakly $J$-quasipolar if and only if one of the
following holds:
\begin{enumerate}
\item Either $A$ or $A-I_2$ or $A+I_2$ is in $J(M_2(R))$.
\item The equation $x^2-(trA)x+detA=0$ has a root in $J(R)$ and a
root in $\mp 1+J(R)$.
\end{enumerate}
\end{thm}

\begin{proof}
The proof is similar to proof of \cite[Theorem 4.9]{CC3}.
\end{proof}

\begin{lem}\cite[Lemma 1.5]{Cvd}
Let $R$ be a commutative domain. Then $A\in M_2(R)$ is an
idempotent if and only if either $A = 0$, or $A = I_2$, or $A =
\begin{bmatrix}a&b\\c&1-a\end{bmatrix}$ where $bc = a-a^2$.
\end{lem}

\begin{prop}
$A\in M_2(\Bbb Z)$ weakly $J$-quasipolar if and only if one of the
following hold.\begin{enumerate}
\item $A$=$\begin{bmatrix}$
  $-a     &  b      \\$

 $ c      &  a-1$
$\end{bmatrix}$ such that $bc=a-a^2$.
\item $A$ is idempotent.
\item
$A$=$\begin{bmatrix}$
  $-a     &  -b      \\$

 $ -c      &  a-1$
 $\end{bmatrix}$ such that $bc=a-a^2$.

 \end{enumerate}
\end{prop}

\begin{proof}
Assume that $A$ is weakly $J$-quasipolar. Since $J(M_2(\Bbb
Z))=0$, proof is clear. Conversely, If $A$=$\begin{bmatrix}$
  $a     &  b      \\$

 $ c      &  1-a$
$\end{bmatrix}$ and $bc=a-a^2$, then $A$ is idempotent. So $A$ is
weakly $J$-quasipolar. Let $A$=$\begin{bmatrix}$
  $-a     &  b      \\$

 $ c      &  a-1$
$\end{bmatrix}$. If idempotent is chosen as $P$=$\begin{bmatrix}$
  $a     &  -b      \\$

 $ -c      &  1-a$
$\end{bmatrix}$, then it is clear. Lately, let
$A$=$\begin{bmatrix}$
  $-a     &  -b      \\$

 $ -c      &  a-1$
$\end{bmatrix}$. The idempotent is chosen as $P$=$\begin{bmatrix}$
  $a     &  b      \\$

 $ c      &  1-a$
$\end{bmatrix}$, it is clear.
\end{proof}

\end{document}